\documentclass{amsart}
\usepackage{amsfonts}
\usepackage{amsmath}
\usepackage{amssymb}

\newtheorem{theorem}{Theorem}
\newtheorem{corollary}{Corollary}
\newtheorem{lemma}{Lemma}
\newtheorem{proposition}{Proposition}
\newtheorem{notation}{Notation}

\newtheorem{remark}{Remark}
\newtheorem{example}{Example}

\begin{document}

\title[Specht property for the $2$-graded identities of $B_m$]{Specht property for the $2$-graded identities of the Jordan superalgebra of a bilinear form}

\author[Silva]{Diogo Diniz}\address{Departamento de Matem\'atica\\ UAME/CCT-UFCG \\ Avenida Apr\'igio Veloso 882\\ 58109-970 Campina Grande-PB, Brazil} \email{diogo@dme.ufcg.edu.br}

\author[Souza]{Manuela da Silva Souza}\address{Departamento de Matem\'atica
\\IM-UFBA \\
Av. Adhemar de Barros, S/N, Ondina \\
40170-110 Salvador-BA, Brazil} \email{manuela.souza@ufba.br}

\keywords{Graded polynomial identity, Jordan algebra, Specht property}

\subjclass[2010]{15A72, 17C05, 17C20}

\begin{abstract}
Let $K$ be a field of characteristic zero and $V$ a vector space of dimension $m>1$ with a nondegenerate symmetric bilinear form $f:V\times V \rightarrow K$. The Jordan algebra $B_m=K\oplus V$ of the form $f$ is a superalgebra with this decomposition. We prove that the ideal of all the $2$-graded identities of $B_m$ satisfies the Specht property and we compute the $2$-graded cocharacter sequence of $B_m$.
\end{abstract}

\maketitle

\section{Introduction}
Let $K$ be a field of characteristic zero and $V$ be a vector space of dimension $m>1$ with a nondegenerate symmetric bilinear form $f$. The Jordan algebra of the form $f$ (see \cite{zsss}) is denoted by $B_m$. The algebra $B_m$ is important in the class of Jordan algebras due to a well-known result of Zelmanov \cite{zelmanov} that classified,  up to isomorphism, finite dimensional simple Jordan algebras when $K$ is algebraically closed in one of the following types: $M_n(K)^+$, the special Jordan algebra of $n\times n$ matrices, $M_n(K)^t$, the algebra of $n\times n$ symmetric matrices with respect to the transpose involution, $M_{2n}(K)^s$, the $2n\times 2n$ symmetric matrices with respect to the symplectic involution, and the algebra $B_m$ when the form is nondegenerate.  These are special simple Jordan algebras, for example, the Clifford algebra of the form $f$ is an associative enveloping for $B_m$. The Jordan algebra $B$ of the form $f$ is defined similarly by replacing a vector space of finite dimension by a vector space of infinite dimension.

Polynomial identities for these algebras have been studied by a number of researchers. Drensky in \cite{drensky} obtained a complete description of the relatively free algebras of the varieties generated by $B_m$ and $B$ denoted by $\mbox{var}(B_m)$ and $\mbox{var}(B)$ respectively. Furthermore he also established the $GL_m$-module structure of these algebras and their Hilbert series. Koshlukov has studied the polynomial identities and the asymptotic behavior of codimensions for the subvarieties of $\mbox{var}(B_m)$ and $\mbox{var}(B)$ (see \cite{koshlukov}). The ideal $\mbox{Id}(A)$ of all identities of an algebra $A$ satisfies the Specht property if $\mbox{Id}(A)$ itself and all $T$-ideals containing $\mbox{Id}(A)$ are finitely generated as $T$-ideals. Kemer proved that every associative algebra satisfies the Specht property (see, for example, \cite{kemer}). Iltyakov \cite{iltyakov} proved  that the variety of unitary algebras generated by $B_m$ satisfy the Specht property. The analogous result for finitely generated Jordan PI-algebras was proved by Va$\breve{i}$s and Zelmanov in \cite{vais/zelmanov}.  In \cite{vasilovsky1} Vasilovsky has explicited a finite bases for polynomial identities of $B$ and $B_m$ for $m \geq 2$ over an infinite field of characteristic $\neq 2$ using ideas developed by Il'tyakov in \cite{iltyakov}.

Other types of identities such as trace identities and graded identities for the algebra $B_m$ were also studied (see \cite{vasilovsky2}, \cite{vasilovsky3}, \cite{koshlukov/silva}). The decomposition $B_m=K\oplus V$ is a natural $\mathbb{Z}_2$-grading and a finite basis for the graded identities was determined in  \cite{koshlukov/silva}. The $\mathbb{Z}_2$-graded identities of the algebras $B_n$ with other types of gradings were studied in \cite{vasilovsky2}. We remark that the group gradings of the algebras $B_n$ were described in \cite{bahturin/shestakov}. This result was generalized by Martino in \cite{martino} for $J_n = B_m \oplus D_k$ the Jordan algebra with a degenerate bilinear form of rank $m$, where $D_k$ is the vector space spanned by the degenerate elements of the basis of a vector space of dimension $n$. For associative algebras graded by a finite group it was proved that every ideal of graded identities is finitely generated (\cite{sviridova}, \cite{aljadeff/belov}). Concerning Lie algebras the Specht property for the graded identities of the Lie algebra $sl_2(K)$ of $2\times 2$ traceless matrices was established  in \cite{razmyslov} for the trivial grading and in \cite{giambruno/souza} for the other gradings.

In the present paper we prove that the variety of $\mathbb{Z}_2$-graded Jordan algebras generated by $B_m=K\oplus V$ satisfies the Specht property and we compute the $2$-graded cocharacter sequence of $B_m$.

\section{Preliminaries}
Throughout the paper $K$ is a field of characteristic zero and all vector spaces and algebras are considered over $K$. Let $V_m$ be a vector space of dimension $m$ with a nondegenerate symmetric bilinear form 
$f:V\times V \rightarrow K$. The vector space $B_m=K\oplus V_m$ with the multiplication defined by \[(\alpha+u)(\beta+v)=(\alpha\beta+ f(u,v))+(\alpha v+\beta u), \hspace{0.5cm} \alpha,\beta \in K, u,v \in V,\] is a Jordan algebra. We set $(B_m)_0=K$, $(B_m)_1=V_m$ and obtain the structure of a Jordan superalgebra on $B_m$. 

Let $Z=X\cup Y$ be the disjoint union of two countable sets. The free Jordan algebra $J(Z)$ freely generated by $Z$ has a $\mathbb{Z}_2$-grading \[J(Z)=J(Z)_0\oplus J(Z)_1,\] where $J(Z)_0$ (resp. $J(Z)_1$) is the subspace generated by the monomials with an even (resp. odd) number of variables in $Y$.

A homomorphism $\Theta: A=A_0\oplus A_1\rightarrow B=B_0\oplus B_1$ of $\mathbb{Z}_2$-graded algebras is \textit{graded} if $\Theta(A_i)\subseteq B_i$, $i\in \mathbb{Z}_2$. A mapping $\theta:Z\rightarrow A$ such that $\theta(X)\subseteq A_0$ and $\theta(Y)\subseteq A_1$ can be uniquely extended to a graded homomorphism  $\Theta: J(Z)\rightarrow A$. The set of polynomials $f(x_1,\dots, x_k, y_1,\dots, y_n)$ such that $f(a_1,\dots, a_k, a^{\prime}_1,\dots, a_n^{\prime})=0$ for every $a_1,\dots, a_k \in A_0$, $a_1^{\prime},\dots, a_n^{\prime}\in A_1$ is an ideal of $J(Z)$. Such ideal, denoted by $\mbox{Id}_2(A)$ is invariant under all graded endomorphisms of $J(Z)$ and is called a \textit{$T_2$-ideal}. An element of $\mbox{Id}_2(A)$ is a \textit{graded polynomial identity} for $A$. If $S\subseteq \mbox{Id}_2(A)$ is a set of polynomials such that $\mbox{Id}_2(A)$ is the intersection of all $T_2$-ideals of $J(Z)$ containing $S$ we say that this set is a \textit{basis} for $\mbox{Id}_2(A)$. The variety of $\mathbb{Z}_2$-graded Jordan algebras determined by $A$ satisfies the \textit{Specht property} if every $T_2$-ideal containing $\mbox{Id}_2(A)$ admits a finite basis.

The $T_2$-ideal $\mbox{Id}_2(A)$ is graded,i.e., 
\[\mbox{Id}_2(A)=\mbox{Id}_2(A)\cap J(Z)_0\oplus \mbox{Id}_2(A)\cap J(Z)_1,\] and therefore the quotient algebra $J(Z)/\mbox{Id}_2(A)$, denoted by $J(A)$, admits a natural grading that makes the quotient map a graded homomorphism. The algebra $J(A)$ is the \textit{relatively free algebra} in the variety determined by $A$. 

The study of the graded identities of an algebra $A$ in characteristic zero can be reduced to the study of the multilinear ones. For every $n + k \geq 1$ and $k \geq 0$ the vector space $P_{k, n}$ of the multilinear polynomials of degree $n + k$ in the variables $x_1$, $\ldots$, $x_k$, $y_1$, $\ldots$, $y_n$ is a $S_k \times S_n$-module under the action
$$(\lambda, \mu)f(x_1, \ldots, x_k, y_1, \ldots, y_n) = f(x_{\lambda(1)}, \ldots, x_{\lambda(k)}, y_{\mu(1)}, \ldots, y_{\mu(n)})$$
and this in turn induces a structure of the $S_k \times S_{n}$-module to the space 
$$P_{k, n}(A) = \dfrac{P_{k, n}}{(P_{k, n} \cap \mbox{Id}_2(A))}.$$
Let
$$\chi_{k, n}(A) = \chi_{k, n}(P_{k, n}(A)) = \sum m_{\sigma, \tau} \chi_\sigma \otimes \chi_\tau$$
be the decomposition of $\chi_{k, n}(A)$ the $(k,n)$-th cocaracter of $A
$ into irreducible  $S_k \times S_n$-characters.
 
We recall some results about finite basis property (or well-quasi-ordering).  For more detailed
background see \cite{higman}.

A relation $a \leq b$ on a set $A$ is a quasi-order if it is reflexive and transitive. If $B$ is a subset of a quasi-ordered set $A,$ the closure of $B$, denoted by $\overline{B},$ is defined as
$\overline{B} = \{a \in A\; : \; \exists\; b \in B \mbox{ such that } b \leq a \}.$
We say that $B$ is a closed subset when $B = \overline{B}$. The quasi-ordered set $(A, \leq)$  has the finite basis property if every closed subset of $A$ is the closure of a finite set. Next we state some alternative definitions of this property. 

\begin{theorem} \label{pbf}The following conditions on a quasi-ordered set $A$ are equivalent.
\begin{enumerate}
\item Every closed subset of $A$ is the closure of a finite subset;
\item If $B$ is any subset of $A,$ there is a finite $B_0$ such that $B_0 \subset B \subset \overline{B_0};$
\item Every infinite sequence of elements $\{a_i\}_{i \geq 0}$ of $A$ has an infinite ascending subsequence $a_{i_1} < a_{i_2} < \cdots < a_{i_k} < \cdots.$
\end{enumerate}
\end{theorem}

The next proposition is immediate using condition $(3)$.

\begin{proposition}\label{ppbf}
If $(A_1, \leq_{1})$, $\ldots$, $(A_k, \leq_k)$ have the finite basis property than $(A_1\times \cdots \times A_k, \leq)$ has the finite basis property where $(a_1, \ldots, a_k)\leq (b_1, \ldots, b_k)$ if and only if $a_i \leq_i b_i$ for all $i \in\{1, \ldots, k\}$.
\end{proposition}

\section{The graded identities of $B_m$}

A basis for the graded identities of $B_m$ was determined in \cite{koshlukov/silva} using the theory developed in 
\cite{concini/procesi}. In this section we recall the necessary results and definitions from these articles.

The polynomials 
\begin{equation}\label{identity}
(x_1 z_1)z_2-x_1(z_1 z_2),
\end{equation}
where $z_1\in \{x_1,y_1\}$ and $z_2\in \{x_2, y_2\}$, are graded identities for $B_m$. Moreover we have the following:

\begin{proposition}\cite[Corollary 21]{koshlukov/silva}\label{basisprop}
The graded identities (\ref{identity}) together with the identity \[\sum_{\sigma \in S_{m+1}} (-1)^{\sigma}y_{\sigma(1)}(y_{m+1}y_{\sigma(2)})\cdots (y_{2m+1}y_{\sigma(m+1)}),\] form a basis for the ideal of graded identities of the Jordan superalgebra $B_m$.
\end{proposition}

A multihomogeneous polynomial $f(x_1, \ldots, x_k, y_1, \ldots, y_n)$ can be written, modulo $\mbox{Id}_2(B_m)$, as $$f(x_1, \ldots, x_k, y_1, \ldots, y_n) = x_1^{\alpha_1} \ldots x_k^{\alpha_k}g(y_1, \ldots, y_n)$$ where $g$ is a polynomial on the variables $y$'s only. 

 Now we present some results from \cite{concini/procesi}. A \textit{double tableau} is an array \[T=\left(
\begin{array}{cccc}
	p_{11}&p_{12}&\cdots&p_{1m_1}\\
	p_{21}&p_{22}&\cdots&p_{2m_2}\\
	\vdots&\vdots&\ddots&\vdots\\
	p_{k1}&p_{k2}&\cdots&p_{km_k}
\end{array}\left|\begin{array}{cccc}
	q_{11}&q_{12}&\cdots&q_{1m_1}\\
	q_{21}&q_{22}&\cdots&q_{2m_2}\\
	\vdots&\vdots&\ddots&\vdots\\
	q_{k1}&q_{k2}&\cdots&q_{km_k}
\end{array}\right. \right),
\]
where $m_1\geq m_2\geq \dots \geq m_k$ and the $p_{ij}$ and $q_{ij}$ are positive integers. An array obtained from a double tableau by replacing $p_{11}$ by $0$ is called a \textit{$0$-tableau}. If $T$ is a double tableau or a $0$-tableau we may form from $T$ the single tableau \[T^{\prime}=\left(\begin{array}{cccc}
	p_{11}&p_{12}&\cdots&p_{1m_1}\\
	q_{11}&q_{12}&\cdots&q_{1m_1}\\
	p_{21}&p_{22}&\cdots&p_{2m_2}\\
	q_{21}&q_{22}&\cdots&q_{2m_2}\\
	\vdots&\vdots&\ddots&\vdots\\
	p_{k1}&p_{k2}&\cdots&p_{km_k}\\
	q_{k1}&q_{k2}&\cdots&q_{km_k}
\end{array}\right).\] We say that $T$ is \textit{doubly standard} if the tableau $T^{\prime}$ is standard. Recall that the single tableau 
\[A=\left(\begin{array}{cccc}
	a_{11}&a_{12}&\cdots&a_{1m_1}\\
	\vdots&\vdots&\ddots&\vdots\\
	a_{k1}&a_{k2}&\cdots&a_{km_k}\\
\end{array}\right)\] is standard if we have the inequalities 
\begin{enumerate}
\item[(i)] $a_{ij}<a_{il}$, if $l>j$
\item[(ii)] $a_{ij}\leq a_{lj}$, if $l\geq i$.
\end{enumerate}
We associate (see \cite{koshlukov/silva}) to a tableau $T$ a polynomial, still denoted by $T$, of $J(B_m)$ \[T=\prod_{l=1}^k det\left|(y_{p_{li}}\cdot y_{q_{lj}})\right|, \mbox{ } i,j=1,\dots, m_l.\] If $T$ is a $0$-tableau the tableau $T_1$ obtained from $T$ by deleting the first line is a double tableau and we associate to $T$ the polynomial 
\[T=\left(\sum_{\sigma \in S_{m_1}}(-1)^{\sigma}y_{q_{1\sigma(1)}}(y_{p_{12}}\cdot y_{q_{1\sigma(2)}})\cdots (y_{p_{1m_1}}\cdot y_{q_{1\sigma(m_1)}}) \right)\cdot T_1.\]
\begin{notation}
We denote by $T_n$ and $T_n^0$ the polynomial associated to the tableau $(12\dots n|12\dots n)$ and the $0$-tableau $(02\dots n|12\dots n)$ respectively.
\end{notation}

Let $A$ be a standard single tableau of the form 
\[A=\left(\begin{array}{ccccccc}
	1&2&3&\cdots&k_1&\tau_1&\dots\\
	1&2&3&\cdots&k_2&\tau_2&\dots\\
	\vdots&\vdots&\vdots&\vdots& & & \\
	1&2&3&\cdots&k_s&\tau_s&\dots\\
	\tau_{s+1}&&&&&&
\end{array}\right),\]
where $\tau_i>k_i+1$, $i=1,\dots, s$ and $\tau_{s+1}>1$. Since $A$ is standard we have 
\begin{equation}\label{k}
k_1\geq k_2\geq \dots \geq k_s.
\end{equation}
Let $h_i(A)$ denote the number of times the number $i-1$ appear in the sequence (\ref{k}) for $i=2,\dots, n$. Thus, for example, $h_2(A)$ is the number of rows starting with $1 \tau$, $\tau>2$, $h_3(A)$ is the number of rows starting with $1 2 \tau$, $\tau > 3$, etc. Consider the tableau obtained from $A$ by deleting the $1$ in each of the $h_2(A)$ rows of $A$ starting with $1 \tau$, $\tau>2$ and replacing it with $2$, then deleting the $1$ in each of the $h_3(A)$ rows starting with $1 2 \tau$, $\tau > 3$, replacing it with a $3$, and so on. If we reorder the elements in each row the resulting tableau, which is denoted by $F(A)$ is standard.

\begin{notation}
Let $T$ be a double tableau and $T^{\prime}$ the corresponding single tableau associated to $T$. We denote by $F(T)$ the double tableau for which the corresponding single tableau is $F(T^{\prime})$.
\end{notation}

Let $\lambda_2,\dots,\lambda_n$ be arbitrary scalars. We substitute in the linear combination \[p=\sum c_iT_i\] of distinct doubly standard tableau the variable $y_1$ by 
$y_1+\sum_{i=2}^n \lambda_i y_i$. The resulting polynomial in $J(B_m)$ is 
\begin{equation}\label{exp}
\sum \lambda_2^{h_2}\lambda_3^{h_3}\cdots \lambda_n^{h_n} P_{h_2\dots h_n}.
\end{equation}
We consider the lexicographical order on the tuples $(h_2,\dots_, h_n)$. The next remark is a simple adaptation of the discussion in \cite[Section 1]{concini/procesi} (see also the proof of Lemma 5.4 in \cite{concini/procesi}).

\begin{remark}\label{remark}
The polynomial in (\ref{exp}) corresponding to the maximal $(h_2,\dots, h_n)$ is \[\overline{p}=\sum_j \epsilon_jc_jF(T_j),\] where $\epsilon_j=\pm 1$, and the sum is taken from the doubly standard tableaux $T_j$ for which $(h_2(T_j^{\prime}),h_3(T_j^{\prime})\dots h_n(T_j^{\prime}))$ is maximal. Moreover the doubly standard tableaux $F(T_j)$ are pairwise distinct.
\end{remark}

\begin{proposition}[\cite{koshlukov/silva}, Proposition 19]\label{basis}
Let $M$ be the subalgebra of $J(B_m)$ generated by the variables in $Y$. The doubly standard tableau form a basis of the vector space $M_0=M\cap J(B_m)_0$. Moreover the doubly standard $0$-tableau form a basis of the vector space $M_1=M\cap J(B_m)_1$.
\end{proposition}

Now we are able to prove our first result. The next remarks will be usefull.

\begin{remark}\label{remark2}
Let $h, h^{\prime}$ be polynomials in $M_1$ not depending on $y_1$. Since the bilinear form in the definition of $B_m$ is nondegenerate the equality $hy_1=h^{\prime}y_1$ implies that $h=h^{\prime}$.
\end{remark}

\begin{remark}\label{remark3}
If $h(y_1,\dots, y_{m+1})\in M_0$ is a polynomial of degree $1$ in $y_{j}$ then there exists $h^{\prime}\in M_1$ such that $h=h^{\prime}y_{j}$. To prove this claim note that by renaming the variables we may assume that $j=1$. We write $h$ as a linear combination of doubly standard tableaux. The standard nature of the tableaux imply that the leftmost entry in the first row is $1$ and this proves the claim. 
\end{remark}

\begin{proposition}\label{prop}
Let $g(y_1,\dots,y_m)$ be a multihomogeneous polynomial of $J(B_m)$ that for any scalars $\lambda_2,\dots, \lambda_m$ is invariant under the substitution $y_1\equiv y_1+\sum_{i=2}^m\lambda_iy_i$. Then there exists a polynomial $h(y_2,\dots,y_m)$ such that one of the following equalities hold:
\begin{enumerate}
\item[(i)] $g(y_1,\dots,y_m)=(T_m)^k h(y_2,\dots,y_m)$,  
\item[(ii)] $g(y_1,\dots,y_m)=(T_m)^{k^{\prime}}T_m^0h(y_2,\dots,y_m)$, 
\end{enumerate}
where $k,k^{\prime}\geq 0$. Moreover if $g$ lies in $J(B_m)_0$ then the equality that holds is $(i)$. 
\end{proposition}
\begin{proof}
We first consider the case $g\in J(B_m)_0$. It follows from the previous proposition that we may write 
\begin{equation}\label{g}
g=\sum c_iT_i, \mbox{ } c_i\neq 0
\end{equation}
as a linear combination of pairwise distinct doubly standard tableau. The invariance of $g$ and Remark \ref{remark} imply that 
$h_2(T_i^{\prime})=h_3(T_i^{\prime})=\dots=h_m(T_i^{\prime})=0$ for every $T_i$ in (\ref{g}). 
Since the variables appearing in $g$ are $y_1,\dots, y_m$ this implies that the entry $1$ only appears in the rows of $T_i$ that are equal to $(12\dots m|12\dots m)$. If $k_i$ is the number of such rows we have $2k_i=deg_{y_1}g$. Let $k$ be the common value of the $k_i$. For every $T_i$ appearing in (\ref{g}) we have $T_i=(T_m)^k\widehat{T_i}$, where $\widehat{T_i}$ is the tableau obtained from $T_i$ by deleting the rows in which the entry $1$ appears. Hence $g$ is decomposed as in $(i)$. 

Assume now that $g \in J(B_m)_1$ and write $gy_{m+1}$ as in (\ref{g}). Note that $h_2(T_i^{\prime})=h_3(T_i^{\prime})=\dots=h_m(T_i^{\prime})=0$ for every $T_i$ appearing $gy_{m+1}$. Hence the entry $1$ only appears in the rows of $T_i$ equal to $(12\dots m| 12\dots m)$ or $(12\dots m|2\dots m+1)$. This last row appears at most once, it appears if $deg_{y_1}g$ is odd and does not appear otherwise. Let $k_i$ be the number of rows in $T_i$ that are equal to $(12\dots m| 12\dots m)$. If $(12\dots m|23\dots m+1)$ does not appear in $T_i$ we have $2k_i=deg_{y_1}g$. Therefore in this case for every $T_i$ appearing in $gy_{m+1}$ we have $T_i=(T_m)^k\widehat{T_i}$, where $k=deg_{y_1}g/2$.  Hence $gy_{m+1}=(T_m)^kh(y_2,\dots, y_{m+1})$ and it follows from Remark \ref{remark2} and Remark \ref{remark3} that $g$ is as in $(i)$. If $(12\dots m|23\dots m+1)$ appears in some $T_i$ we have $2k_i+1=deg_{y_1}g$ and thus it appears in every $T_i$. The polynomial associated to the tableau $(12\dots m|23\dots m+1)$ is $\pm T_m^0y_{m+1}$. In this case $gy_{m+1}=((T_m)^{k}T_m^0h(y_2,\dots, y_m))y_{m+1}$, where $k=(deg_{y_1}g-1)/2$. From this last equality and Remark \ref{remark2} we conclude that $g$ is as in $(ii)$.
\end{proof}

\section{The main results}

We denote by $M_{m}$ the subalgebra of $J(B_m)$ generated by $y_1,\dots, y_m$. Let $W$ denote the subspace of $J(B_m)$ generated by $y_1,\dots, y_m$ and $GL_m$ the group of invertible linear transformations of $W$. The canonical action of $GL_m$ on $M_{m}$ turns this subalgebra into a $GL_m$-module. We refer the reader to \cite{drensky/formanek} for basic facts regarding polynomial representations of $GL_m$ and applications to algebras with polynomial identities.

\begin{lemma}\cite[Theorem 2.2.11]{drensky/formanek}\label{lemma}
Let $g(y_1,\dots, y_m)$ be a multihomogeneous polynomial of $M_m$. The $GL_m$-module generated by $g(y_1,\dots, y_m)$ is an irreducible module if and only if for any scalars $\lambda_{ij}$ $1\leq i<j\leq m$ it is invariant under the substitutions 
\begin{equation}\label{subs0}
y_j\equiv \lambda_{1j}y_1+\dots +\lambda_{j-1j}y_{j-1}+y_j\mbox{ } j=1,2,\dots,m.
\end{equation}
\end{lemma}

Next we apply this lemma and Proposition \ref{prop} to determine generators of the irreducible submodules of $M_m$. The following notation will be usefull.

\begin{notation}
Given a polynomial $g(y_1,\dots,y_m)$ in $M_m$ we denote $\widehat{g}(y_1,\dots,y_m)$ the polynomial $g(y_m,\dots, y_1)$. Denote by $S_m$ the polynomial corresponding to the tableaux $(012\dots m-1|12 \dots m)$. We have $\widehat{T_m^0}=\pm S_m$.
\end{notation}

Note that the polynomial $g(y_1,\dots,y_m)$ in $M_m$ is invariant under the substitutions 

\[y_m\equiv \lambda_1y_1+\dots \lambda_{m-1}y_{m-1}+y_m,\]

for any $\lambda_i \in K$ if and only if $\widehat{g}(y_1,\dots,y_m)=g(y_m,\dots, y_1)$ is invariant under the substitutions 

\begin{equation}\label{subs}
y_1\equiv y_1+\sum_{i=2}^m\lambda_iy_i,
\end{equation}

for any $\lambda_i \in K$. 

\begin{corollary}\label{theorem1}
The polynomial $g(y_1,\dots,y_m)\in M_m$ generates an irreducible $GL_m$-submodule of $M_m$ if and only if it is up to multiplication by a scalar one of the polynomials
\begin{equation}\label{pol}
(S_{m})^{\delta_{m}} \cdots (S_{1})^{\delta_{1}} (T_m)^{k_m}\cdots (T_1)^{k_1}
\end{equation}
where $k_1$, $\dots$,  $k_m \geq 0$,  $\delta_l\in \{0,1\}$  and $\delta_l \neq 0$ for at most one index $l \in \{1, \ldots, m\}$.
\end{corollary}

\begin{proof}
We prove the result by induction on $m$. If $m=1$ the result is obvious. Now let $g(y_1,\dots,y_m)\in M_m$ generate an irreducible submodule. If follows from Lemma \ref{lemma} that $\widehat{g}$ is invariant under the substitutions (\ref{subs}). It follows from Proposition \ref{prop} that we have two possibilities: (i) $\widehat{g}=(T_m)^k h(y_2,\dots,y_m)$ or (ii) $\widehat{g}=(T_m)^{k^{\prime}}T_m^0h(y_2,\dots,y_m)$. In the first case we obtain $g=(T_m)^k h^{\prime}(y_1,\dots,y_{m-1})$. It is clear that $h^{\prime}$ is invariant under the substitutions (\ref{subs0}) and the result follows.  In the second case $g=\pm (T_m)^{k^{\prime}}S_mh^{\prime}(y_1,\dots,y_{m-1})$ with $h^{\prime}$ invariant under the substitutions (\ref{subs0}). Note that this occurs only if $g$ lies in $J(B_m)_1$ and in this case $h^{\prime}$ lies in $J(B_m)_0$. The induction hypothesis applied to $h^{\prime}$ implies that it is up to scalar of the form $(T_{m-1})^{k_{m-1}}\cdots (T_1)^{k_1}$. 
\end{proof}

\begin{example} The algebra $B_2$ of $2 \times 2$ symmetric matrices over $K$  with the Jordan product $a \circ b = (ab+ ba)/2$ is a Jordan algebra of a nondegenerate symmetric bilinear form on a vector space of dimension $2$. In this case $g(y_1,y_2) \in M_2$ generates an irreducible $GL_2$-submodule of $M_2$ if and only if it is up to multiplication by a scalar one of the polynomials
$$(y_1^2 y_2^2 - (y_1y_2)^2)^{k_2} y_1^{2k_1} \mbox{ if } g \in J(B_2)_0;$$ 
$$(y_1^2 y_2^2 - (y_1y_2)^2)^{k_2} y_1^{2k_1 + 1} \mbox{ or } (y_1^2 y_2^2 - (y_1y_2)^2)^{k_2} (\bar{y}_1 y_1\bar{y}_2)y_1^{2k_1} \mbox{ if } g \in J(B_2)_1.$$  
\end{example}

Recall that any result on homogeneous polynomial identities obtained in the language of representations of the general linear group is equivalent to a corresponding result on multilinear polynomial identities obtained in the language of representation of the symmetric group. The complete linearization of a highest weight vector associated to an irreducible $GL_1\times GL_m$-module generates an irreducible $S_k \times S_n$-module. As a consequence we have the description of the $2$-graded cocharacter sequence of $B_m$.


\begin{theorem}  Let 
$$\chi_{k, n}(B_m) = \sum_{\lambda \vdash k,\, \mu \vdash n}m_{\lambda, \mu} \chi_{\lambda}\otimes\chi_{\mu}$$
be the $(k, n)$-th cocharacter of $B_m$. Then, for every $\lambda \vdash k \mbox{ and } \mu \vdash n,$ $m_{\lambda, \mu} \leq 1$. Moreover, $m_{\lambda, \mu} = 1$ if and only if $\lambda = (k),$ $\mu = (r_m, \ldots, r_1)$ is a partition of $n$ such that at most one $r_i$ is odd.
\end{theorem}

Now we proceed to prove that $B_m$ has the Specht property. We will need the following:

\begin{lemma}
Let $\widetilde{T_k}$, $\widetilde{S_k}$, $k=1,2,\dots$ be polynomials in $J(Z)$ such that the image of $\widetilde{T_k}$, $\widetilde{S_k}$ under the canonical homomorphism are the polynomials $T_k$, $S_k$ in $J(B_m)$ respectively. The set of polynomials 
\begin{equation}\label{polynomials}
x_1^{k_0}(\widetilde{S_{m}})^{\delta_{m}} \cdots (\widetilde{S_{1}})^{\delta_{1}} (\widetilde{T_{m}})^{k_m}\cdots (\widetilde{T_{1}})^{k_1},
\end{equation} 
where $k_0$, $\dots$,  $k_m \geq 0$,  $\delta_l\in \{0,1\}$  and $\delta_l \neq 0$ for at most one index $l \in \{1, \ldots, m\}$, has the following property: given any subset $S$ there exists a finite subset $\widehat{S}\subseteq S$ such that any polynomial in $S$ is the consequence of a polynomial in $\widehat{S}\cup Id_2(B_m)$.
\end{lemma}
\begin{proof}
To prove this we define in $\mathbb{N}^{2m+1}$ the partial order $\leq$ as follows:
\[(\delta_1, \dots, \delta_m, k_0,\dots, k_m)\leq (\delta_1^{\prime}, \dots, \delta_m^{\prime}, k_0^{\prime},\dots, k_m^{\prime})\mbox{ if } \delta_i \leq \delta_i^{\prime}, k_j\leq k_j^{\prime} \mbox{ for all } i, j.\] Clearly the above inequality implies that $x_1^{k_0^{\prime}}(S_{m}^0)^{\delta_{m}^{\prime}} \cdots (S_{1}^0)^{\delta_{1}^{\prime}} (T_{m})^{k_m^{\prime}}\cdots (T_{1})^{k_1^{\prime}}$ is obtained from $x_1^{k_0}(S_{m})^{\delta_{m}} \cdots (S_{1})^{\delta_{1}} (T_{m})^{k_m}\cdots (T_{1})^{k_1}$ by multiplication by a suitable element of $J(B_m)$. By Proposition \ref{ppbf} the set $\mathbb{N}^{2m+1}$ with the partial order above has the finite basis property and this implies the result.
\end{proof}

Next we prove our main result.

\begin{theorem}\label{main}
The ideal $Id_2(B_m)$ of $\mathbb{Z}_2$-graded identities of $B_m=K\oplus V$ has the Specht property.
\end{theorem}
\begin{proof}
Let $I$ be a $T_2$-ideal containing $Id_2(B_m)$. Henceforth we say that two sets of polynomials $S$ and $S^{\prime}$ in $J(Z)$ are equivalent if $S\cup Id_2(B_m)$ and $S^{\prime}\cup Id_2(B_m)$ generate the same $T_2$-ideal. We claim that for every multilinear polynomial $g$ in $I$ there exists a set of polynomials $S_g\subseteq I$ of polynomials of the form (\ref{polynomials}) such that $\{g\}$ and $S_g$ are equivalent. Let $S=\cup S_g$, where the union is over all multilinear polynomials $g$ in $I$. Since the field is of characteristic zero this implies that $I$ is generated by $S\cup Id_2(B_m)$. Hence the previous lemma and Proposition \ref{basisprop} imply the theorem. Let $g(x_1,\dots, x_k,y_1,\dots, y_n)$ be a multilinear element in $I$. We work in the relatively free algebra $J(B_m)$ and denote by $\varphi:J(Z)\rightarrow J(B_m)$ the canonical homomorphism. The group $S_n$ acts on the odd variables of $\varphi(g)$. We decompose this module, denoted by $M$, into a direct sum of irreducible modules \[M=M_1\oplus\dots\oplus M_q,\] and let $p_i$ be a generator of the $S_n$-module $M_i$. A set of $q$ polynomials $\{P_1,\dots, P_q\}$ such that $\varphi(P_i)=p_i$ is equivalent to $\{g\}$. The result is proved once we prove that each $P_i$ is equivalent to a finite set of polynomials of the form  (\ref{polynomials}). Let $\lambda\vdash n$ be the partition of $n$ corresponding to $M_i$. We write $p_i=x_1\cdots x_{k}p_i^{\prime}(y_1,\dots,y_n)$ and note that the $S_n$-module generated by $p_i^{\prime}$ is isomorphic as an $S_n$-module to $M_i$. Since $dim V=m$ we conclude that $\lambda_{m+1}=0$. Then $p_i^{\prime}$ is symmetric in $m$ disjoint sets of variables and we may identify the variables in each set to obtain a multihomogeneous polynomial $q_i(y_1,\dots, y_m)$. The complete linearization of $q_i$ is $p_i^{\prime}$. We decompose the $GL_m$-module generated by $q_i$ and let $q_i^1,\dots,q_i^r$ be generators of its irreducible components. Thus Corollary \ref{theorem1} implies that each $q_i^j$ is up to scalar a polynomial of the form (\ref{pol}) and hence $x_1^{k}q_i^j$ is the image under the canonical homomorphism of a polynomial $Q_i^j$ of the form (\ref{polynomials}). The linearization polynomial $x_1^{k}q_i$ is up to multiplication by scalar $p_i$. The set $S_i=\{Q_i^1\dots Q_i^r\}$ is equivalent to $\{P_i\}$. Hence $S_g=S_1\cup \dots \cup S_q$ is equivalent to $\{g\}$.
\end{proof}

\end{document}